\documentclass{amsart}

\usepackage{amsmath,amssymb,amsfonts,url,epsf,epsfig}
\usepackage{euscript,graphicx,float}

\usepackage{hyperref, color, ulem}

\newtheorem{theorem}{Theorem}[section]
\newtheorem{lemma}[theorem]{Lemma}

\newcommand{\Dart}{{\rm D}}

\newcommand{\V}{{\mathcal{V}}}
\newcommand{\D}{{\mathcal{D}}}

\newcommand{\dAAG}{\mathrm {A^2D}}

\begin{document}

\title[Connectedness]{Connectedness of the dart digraph and the squared-dart digraph}

\author[P.\ Poto\v{c}nik]{Primo\v{z} Poto\v{c}nik}
\address{Primo\v{z} Poto\v{c}nik,\newline
Faculty of Mathematics and Physics, University of Ljubljana, \newline Jadranska 19, SI-1000 Ljubljana, Slovenia;
\newline also affiliated with: \newline 
 IMFM,
 Jadranska 19, SI-1000 Ljubljana, Slovenia.
 } 
\email{primoz.potocnik@fmf.uni-lj.si}

\thanks{Supported in part by the Slovenian Research Agency, projects J1-5433, J1-6720, and P1-0294}


\author[S.\ Wilson]{Steve Wilson}
\address{Steve Wilson,\newline
 Northern Arizona University, Department of Mathematics and Statistics, \newline
 Box 5717, Flagstaff, AZ 86011, USA
}
 \email{stephen.wilson@nau.edu}


\begin{abstract}
In this note we revisit the
 {\it dart graph} and {\it the squared dart digraph} constructions
 and prove that they yield strongly connected digraphs  when 
applied to connected graphs of minimum valence at least $3$.
\end{abstract}

\maketitle


\section{Introduction}
\label{sec:intro}

In \cite{HC} and \cite[Section 4]{PW}, two constructions, called a {\it dart digraph} and a {\it squared dart digraph}, were introduced.  The second of these is a directed form of a graph introduced in  \cite{PSV}.
The purpose of this note is to prove that these two constructions yield strongly connected digraphs whenever applied
to connected graphs.

All the graphs and digraphs in this note are considered simple. More precisely,
we define a {\it digraph} to be a pair $(\V, \D)$ in which $\V$ is a finite non-empty collection of things called {\it vertices} 
and $\D$ is a collection of ordered pairs of distinct vertices.  An element $(u,v)$ of $\D$ will be called a {\it dart} with {\it initial}
vertex $u$ and {\it terminal} vertex $v$.
A {\it $2$-dart} of a digraph $(\V, \D)$ is a pair $(x,y)$ of darts in $\D$ such that the terminal vertex of $x$ coincides
  with the initial vertex of $y$ while the initial vertex of $x$ does not coincide with the terminal vertex of $y$.

If for every dart $(u,v)$ of a digraph $\Lambda$ also its reverse $(u,v)^{-1} = (v,u)$ is a dart, then $\Lambda$ is called a {\it graph}.
In this case, we call the pair $\{(u,v), (v, u)\}$ of darts an {\it edge} of $\Lambda$.

 We are now ready to define the {\it dart digraph} and the {\it squared dart digraph} of a given graph $\Lambda$ with the set of darts $\D$.

The {\it dart digraph of $\Lambda$} is the digraph  $\Dart(\Lambda)$
 with vertices and darts being the darts and $2$-darts of $\Lambda$, respectively. 

Similarly, let the {\it squared dart digraph} of $\Lambda$ be the digraph
$\dAAG(\Lambda)$  with vertex-set $\D\times\D$ and with a pair $\bigl( (x,y),(z,w) \bigr)$, $x,y,z,w\in \D$, being a dart of
$\dAAG(\Lambda)$ if and only if $y=z$ and $(x,w)$ is a $2$-dart of $\Lambda$.

Recall that a digraph is said to be {\it strongly connected}   provided that for any two vertices $u, v$, 
there is a directed path from $u$ to $v$ (we then say that $v$ is {\it accessible} from $u$), 
as well as one from $v$ to $u$. 

\section{Results}
\label{sec:AAG}

The first of our results is a simple observation about bipartiteness of the dart digraph and the squared dart digraph.
(A digraph  is said to be {\it bipartite} if its underlying graph is bipartite.)

\begin{lemma}
\label{lem:bip}
If $\Lambda$ is a bipartite graph, then $\Dart(\Lambda)$ and $\dAAG(\Lambda)$ are also bipartite.
\end{lemma}

\begin{proof}
Colour the vertices of $\Lambda$ properly black and white and let 
   a dart of $\Lambda$ inherit the colour of its initial vertex;
   this then determines a proper colouring of the vertices of $\Dart(\Lambda)$; in particular, $\Dart(\Lambda)$ is bipartite.
   
    Further, colour
   a vertex $(x,y)$ of $\dAAG(\Lambda)$ {\it blue} if the darts $x,y$ of $\Lambda$ are of the same colour
   as vertices in $\Dart(\Lambda)$
    (either black or white),
   and {\it red} otherwise. This is then clearly a proper colouring of the vertices of $\dAAG(\Lambda)$.
\end{proof}

We will now introduce a few auxiliary notions needed to analyse connectedness of the dart digraph and the square dart digraph.

An {\it $s$-arc} in a graph $\Lambda$ is a walk of length $s$ in which no two of any three consecutive vertices are the same;
alternatively, it is a sequence of darts in $\Lambda$ such that any two consecutive darts form a $2$-dart.

An {\it arc-cycle} is a closed walk which is also an $s$-arc for some $s$, and in addition, if it begins with $(a,b)$ and
ends with $(c, a)$, then $c$ is required to be different from $b$.  Note that any cyclic shift of an arc-cycle is also an arc-cycle.
Observe that an $s$-arc in $\Lambda$ corresponds to a directed walk in $\Dart(\Lambda)$ of length $s-1$, 
and an arc-cycle $\Lambda$ of length $s$ corresponds to a directed closed walk of length $s$ in $\Dart(\Lambda)$.

An $s$-arc, written as a sequence  $[a_0, a_1, a_2, \dots, a_{s-1},a_s]$ of vertices,  is a {\it balloon}
 if $a_0,a_1, \dots, a_{s-1}$ are pairwise distinct and $a_s = a_i$ for some $i\in\{1,2,\ldots, s-3\}$.
The arc $(a_0,a_1)$ is then called the {\it beginning} of the balloon.

\begin{lemma}
\label{lem:balloon}
Let $\Lambda$ be a graph in which every vertex has valence at least $3$ and let $(u,v)$ be a dart of $\Lambda$.
Then  $(u,v)$ is the beginning of some balloon in $\Lambda$.
\end{lemma}

\begin{proof}
Let  $\Lambda'$ be the connected component containing $v$ of the graph 
obtained from $\Lambda$ by removing the vertex $u$ and all of the edges incident to $u$.
$\Lambda'$ is not a tree. Hence $\Lambda'$  contains a cycle, say $C=a_0a_1 \ldots a_k$ with $a_k = a_0$.
Let $v v_1\ldots v_m$ be a path from $v$ to $C$ in $\Lambda'$.
 Without loss of generality we may assume that $v_m=a_0$.
Then $[u,v,v_1,\ldots, v_m, a_1, a_2,\ldots, a_k]$ is a baloon in $\Lambda$ starting with $(u,v)$.
\end{proof}

\begin{lemma}
\label{lem:GCD}
Let $\Lambda$ be a graph in which every vertex has valence at least 3. Then
the greatest common divisor of the lengths of all arc-cycles in $\Lambda$ is at most 2.
\end{lemma}

\begin{proof}
Let $C$ be a cycle in $\Lambda$, let $m$ be its length, let $uv$ be an edge of $C$, let $a$ be a neighbour of $u$ other than its neighbours in the cycle $C$, and let $b$ be that for $v$.  Let $\alpha, \beta$ be  balloons beginning with $(u,a)$ and $(v, b)$, respectively.  Then the walk beginning at $u$, following $\alpha$ out to and around  its cycle and back to $u$
along the initial part of $\alpha$, then in one step to $v$, then following $\beta$ out to and around  its cycle and back to $v$ following the initial part of $\beta$, then finally from $v$ back in one step to $u$ is an arc-cycle $\gamma$ of some length $n$.   Replacing that last step from $v$ to $u$ by   the path formed  from $C$  by removing the edge  $\{u,v\}$ gives an arc-cycle of length $m+n-2$.  As the greatest common divisor of $m, n$, and $m+n-2$ is at most 2, the  result follows.
\end{proof}

\begin{theorem}
\label{th:strong}
If $\Lambda$ is a connected simple graph in which every vertex has valence at least 3,  then $\Dart(\Lambda)$ and $\dAAG(\Lambda)$ are strongly connected.
\end{theorem}

\begin{proof}
Let $\Delta = \Dart(\Lambda)$.
We  begin the proof of the strong connectivity of $\Delta$  by proving two claims:
\smallskip

{\sc Claim 1}:  {\it Let $x=(u,v)$ be a dart of $\Lambda$ and let $x^{-1}=(v,u)$ be its inverse dart.
Then there exists a directed walk from $x$ to $x^{-1}$ in $\Delta$.}
\smallskip

\noindent
Indeed:  By Lemma~\ref{lem:balloon}, there exists a balloon 
$\alpha = [a_0, a_1, a_2, \dots, a_{s-1},a_s]$ in $\Lambda$, beginning with $x$ (that is, $a_0=u$ and $a_1=v$). 
Let $i \in \{1, \ldots, s-2\}$ be such that $a_s = a_i$.
Then $\beta = [a_0, a_1, a_2, \dots, a_{s-1}, a_s = a_i, a_{i-1}, a_{i-2}, \dots, a_2, a_1, a_0]$ is an $(s+i)$-arc in $\Lambda$,
yielding a directed walk 
from $x$ to $x^{-1}$ in $\Delta$. This proves Claim 1.
\smallskip

{\sc Claim 2}: {\it If $e$ and $f$ are two edges in $\Lambda$, then there exists a directed walk in $\Delta$ from  some $x$ to some $y$
such that the underlying edges of $x$ and $y$ are $e$ and $f$, respectively.}
\smallskip

\noindent 
To prove this, consider a shortest path $va_1a_2\ldots a_kw$ from $e$ to $f$. Then $e=\{u,v\}$ and $f=\{w,z\}$ for some
vertices $u$ and $z$ of $\Lambda$ such that  $a_1 \not = u$ and $a_k\not = z$. But then
$(u,v)\, (v,a_1)\, (a_1,a_2)\, \ldots\, (a_{k-1},a_k)\, (a_k,w)\, (w,z)$
is a directed walk in $\Delta$  from $x = (u,v)$ to $y =  (w,z)$, underlying $e$ and $f$ respectively.  This proves Claim 2.
\smallskip

Note that strong connectivity of $\Delta$ now follows directly from Claims 1 and 2. Namely, if
$x$ and $y$ are two vertices in $\Delta$ (and thus darts in $\Lambda$), then Claim 2 implies existence
of a directed walk in $\Delta$ from either $x$ or $x^{-1}$ to either $y$ or $y^{-1}$. By inserting directed walks
(the existence of which is implied by Claim 1) 
from $x$ to $x^{-1}$ and $y^{-1}$ to $y$, if necessary, one obtains a directed walk in $\Delta$
from $x$ to $y$. 
\smallskip

Now we are ready to prove that $\dAAG(\Lambda)$ is strongly connected.  Let $(x,y)$ and $(w,z)$ be any two vertices in $\dAAG(\Lambda)$.
Then $x,y,w$ and $z$ are darts of $\Lambda$ and hence vertices of $\Delta$. Since $\Delta$ is strongly connected,
 there are directed walks from $x$ to $w$ and from $y$ to $z$,
  and moreover, we may choose these two walks  so that each passes   through every vertex of $\Delta$

By Lemma~\ref{lem:GCD},
the greatest common divisor $D$ of the lengths of all arc-cycles in $\Lambda$ is at most $2$.  Thus, by inserting arc-cycles appropriately, we can cause the length of the two walks to differ by   at most 1.   Let these  walks    be $\alpha = [x=a_0, a_1, \dots, a_k = w]$ and $\beta = [y=b_0, b_1, \dots, b_\ell = z]$ where $|k-\ell | \le 1$.
  Here, each $a_i$ and $b_i$ is a dart in $\Lambda$ and each $(a_i, a_{i+1})$ and $(b_i, b_{i+1})$ is a 2-dart. 

If $\Lambda$ is not bipartite, then  $D = 1$, and we can force $k$ to be equal to $\ell$.   Then the sequence 
$$(a_0, b_0), (b_0, a_1), (a_1, b_1), (b_1, a_2), \dots, (a_k, b_k)$$
 is a directed walk of length $2k$ from $(x, y)$ to $(w, z)$.

Now suppose that $\Lambda$ is bipartite.   Recall that (see Lemma~\ref{lem:bip}) that then also vertices $\dAAG(\Lambda)$ can be
properly bi-coloured blue and red where a vertex $(x,y)$ is coloured blue whenever the initial vertices of $x$ and $y$ are at even distance
in $\Lambda$. Since every vertex in $\dAAG(\Lambda)$ has positive in- and out-valence,
to prove that $\dAAG(\Lambda)$ is strongly connected it suffices to show that every blue vertex is 
accessible from any other blue vertex, hence we may assume that the vertices $(x,y)$ and $(w,z)$ are blue.
But then the directed walks $\alpha$ and $\beta$ from $x$ to $w$ and from $y$ to $z$ must have the same parity.    Thus, even though $D = 2$, we can again force $k=\ell$, yielding a directed walk of length $2k$ from $(x, y)$ to $(w, z)$, as above.
\end{proof}


\begin{thebibliography}{99}


\bibitem{HC} A.\ Hill, S.\ Wilson,
Four constructions of highly symmetric graphs,
\textit{J.\ Graph Theory} {\bf 71} (2012),  229--244. 

\bibitem{PSV} P.\ Poto\v{c}nik, P.\ Spiga, G Verret, Bounding the order of the vertex-stabiliser in $3$-valent vertex-transitive and $4$-valent arc-transitive graphs,     
{\it J.\ Combin.\ Theory, Ser. B.} {\bf 111} (2015), 148--180.

\bibitem{PW} P.\ Poto\v{c}nik, S.\ Wilson,
The separated box product of two digraphs,
\textit{to be put on arXiv}.


\end{thebibliography}
\end{document}